\documentclass[12pt,oneside,english]{amsart}
\usepackage[T1]{fontenc}
\usepackage[latin9]{inputenc}
\usepackage{mathrsfs}
\usepackage{amsthm}
\usepackage{amstext}
\usepackage{amssymb}
\usepackage{esint}

\makeatletter
\theoremstyle{plain}
\newtheorem{thm}{\protect\theoremname}[section]
  \theoremstyle{plain}
  \newtheorem{lem}[thm]{\protect\lemmaname}
  \theoremstyle{remark}
  \newtheorem{rem}[thm]{\protect\remarkname}
  \theoremstyle{plain}
  \newtheorem{cor}[thm]{\protect\corollaryname}
  \theoremstyle{plain}
  \newtheorem{conjecture}[thm]{\protect\conjecturename}

\usepackage{fullpage}

\makeatother

\usepackage{babel}
  \providecommand{\conjecturename}{Conjecture}
  \providecommand{\corollaryname}{Corollary}
  \providecommand{\lemmaname}{Lemma}
  \providecommand{\remarkname}{Remark}
\providecommand{\theoremname}{Theorem}

\begin{document}

\title{Von Neumann Algebras of Sofic Groups with $\beta_{1}^{(2)}=0$ are
Strongly $1$-Bounded.}

\author{Dimitri Shlyakhtenko}

\address{Department of Mathematics, UCLA, Los Angeles, CA 90095, USA}

\email{shlyakht@math.ucla.edu}

\thanks{Research supported by NSF grant DMS-1500035.}
\begin{abstract}
We show that if $\Gamma$ is a finitely generated finitely presented
sofic group with zero first $L^{2}$ Betti number, then the von Neumann algebra $L(\Gamma)$
is strongly $1$-bounded in the sense of Jung. In particular, $L(\Gamma)\not\cong L(\Lambda)$
if $\Lambda$ is any group with free entropy dimension $>1$, for
example a free group. The key technical result is a short proof of
an estimate of Jung \cite{jung:BettiBounded} using non-microstates
entropy techniques.
\end{abstract}

\maketitle

\section{Introduction.}

In a series of papers \cite{dvv:entropy2,dvv:entropy3}, Voiculescu
introduced the notion of free entropy dimension, as an analog of Minkowski
content in free probability theory. If $X_{1},\dots,X_{n}\in(M,\tau)$
are a self-adjoint $n$-tuple of elements in a tracial von Neumann
algebra, $\delta_{0}(X_{1},\dots,X_{n})$ is a measure of ``how free
they are''. If $M=W^{*}(X_{1},\dots,X_{n})$ satisfies the Connes
embedding conjecture and is diffuse, $1\leq\delta_{0}(X_{1},\dots,X_{n})\leq n$.
The value $1$ is achieved e.g. when $M$ is hyperfinite \cite{jung:dimHyperfinite},
while a free semicircular $n$-tuple generating the free group factor
$L(\mathbb{F}_{n})$ has free entropy dimension $n$ \cite{dvv:entropy3}. 

The number $\delta_{0}(X_{1},\dots,X_{n})$ is an invariant of the
(non-closed) $*$-algebra generated by $X_{1},\dots,X_{n}$. In particular,
if $\Gamma$ is a finitely-generated discrete group, free entropy
dimension of any generating set of the group algebra gives us a group
invariant, $\delta_{0}(\mathbb{C}\Gamma)$ \cite{dvv:improvedrandom}.
As it turns out, $\delta_{0}(\mathbb{C}\Gamma)$ has a close relationship
with the first $L^{2}$ Betti number of $\Gamma$ (we refer the reader
to \cite{luck:book} for background on $L^{2}$-Betti numbers). Indeed,
$\delta_{0}(\mathbb{C}\Gamma)\leq\beta_{1}^{(2)}(\Gamma)+1$ for any
finitely generated infinite group $\Gamma$ \cite{connes-shlyakht:l2betti};
and if $\delta_{0}$ is replaced by its ``non-microstates analog''
$\delta^{*}$, then equality holds \cite{shlyakht-mineyev:freedim}.

Free entropy dimension has found a large variety of applications in
von Neumann algebra theory; for example, it was used by Voiculescu
to prove that free group factors $L(\mathbb{F}_{n})$ have no Cartan
subalgebras \cite{dvv:entropy3}. By now, there is a long list of
various properties that imply that any generators of a von Neumann
algebra $M$ have free entropy dimension $1$; we refer the reader
to Voiculescu's survey \cite{dvv:entropysurvey} and to \cite{jung:onebounded,hayes:oneBounded}. 

In \cite{jung:onebounded} Jung discovered a technical strengthening
of the statement that $\delta_{0}(X_{1},\dots,X_{n})=1$, called \emph{strong
$1$-boundedness}, which we will now briefly review. 

The free entropy dimension $\delta_{0}$ is defined by the formula
\cite{dvv:entropy3}
\[
\delta_{0}(X_{1},\dots,X_{n})=n-\limsup_{\epsilon\to0}\frac{\chi(X_{1}+\epsilon^{\frac{1}{2}}S_{1},\dots,X_{n}+\epsilon^{\frac{1}{2}}S:S_{1},\dots,S_{n})}{\log\epsilon^{\frac{1}{2}}}
\]
where $\chi$ stands for free entropy and $S_{1},\dots,S_{n}$ is
a free semicircular $n$-tuple, free from $X_{1},\dots,X_{n}$. The
statement that $\delta_{0}(X_{1},\dots,X_{n})\leq1$ then translates
into the estimate $\chi(X_{1}+\epsilon^{\frac{1}{2}}S_{1},\dots,X_{n}+\epsilon^{\frac{1}{2}}S_{n}:S_{1},\dots,S_{n})\leq(n-1)\log\epsilon^{\frac{1}{2}}+f(\epsilon)$
with $\limsup_{\epsilon}|f(\epsilon)/\log\epsilon|=0$. Strong $1$-boundedness of the set $X_1,\dots,X_n$ 
is a strengthening of this inequality: it requires that 
\[
\chi(X_{1}+\epsilon^{\frac{1}{2}}S_{1},\dots,X_{n}+\epsilon^{\frac{1}{2}}S_{n}:S_{1},\dots,S_{n})\leq(n-1)\log\epsilon^{\frac{1}{2}}+\textrm{const}
\]
for small $\epsilon$.

Jung proved the following amazing result: If $X=(X_{1},\dots,X_{n})$
is a strongly $1$-bounded set and $\chi(X_j)>-\infty$ for at least one $j$, then any other finite set  $Y=(Y_{1},\dots,Y_{m})\in W^{*}(X)$ satisfying $W^*(Y)=W^*(X)$ is also strongly $1$-bounded; thus in particular, $\delta_{0}(Y_{1},\dots,Y_{m})\leq 1$.

In his follow-up to Jung's work, Hayes \cite{hayes:oneBounded} improved this statement: if $X=(X_1,\dots,X_n)$ is strongly $1$-bounded and either $W^*(X)$ is amenable or $X$ is a non-amenability set, then any $Y=(Y_1,\dots,Y_m)$ with $W^*(Y)\cong W^*(X)$ must also be strongly $1$-bounded.  

We will say that a von Neumann algebra $M$ is strongly $1$-bounded if  $M=W^*(X)$  for some finite strongly $1$-bounded set $X$ which satisfies either of the conditions:  (a) $\chi(X_j)>-\infty$ for some $j$,  (b)  $W^*(X)$ is amenable, or (c)  $X$ is a non-amenability set.  Thus if $M$ is strongly $1$-bounded and $M=W^*(Y)$ for a finite set $Y$, then $Y$ is strongly $1$ bounded and in particular $\delta_0(Y)\leq 1$.

Jung's result means that strong $1$-boundedness of a von Neumann algebra can be checked on a set of its generators. With
few exceptions such as property (T) \cite{hlyakht-jung:freeEntropyPropertyT},
all known implications stating that some property (presence of a Cartan
subalgebra, tensor product decomposition, property $\Gamma$, etc.)
entails $\delta_{0}=1$ for any generating set can be upgraded to
state that the von Neumann algebra is actually strongly $1$-bounded.
We refer the reader to \cite{jung:onebounded,hayes:oneBounded} and
references therein.

The main theorem of this paper states that if a finitely generated
finitely presented sofic group $\Gamma$ satisfies $\beta_{1}^{(2)}(\Gamma)=0$, then the associated von Neumann
algebra is strongly $1$-bounded. In particular, any generating set
of the von Neumann algebra $L(\Gamma)$ has free entropy dimension
$1$. This implies a number of free indecomposability properties of
$L(\Gamma)$. For example, $L(\Gamma)\not\cong M*N$ with $M,N$ diffuse
von Neumann algebras; in particular, $L(\Gamma)\not\cong L(\mathbb{F}_{n})$
for any $n\in[2,+\infty]$. 

To obtain our result, we first give a new proof to a recent result
of Jung \cite{jung:BettiBounded} which allows one to deduce $\alpha$-boundedness
results for free entropy dimension under the assumption of existence
of algebraic relations. Jung's estimate can be considered an improved
version of the free entropy dimension estimates from \cite{connes-shlyakht:l2betti,shlyakht-charlesworth:noAtoms}
that were previously obtained using non-microstates free entropy methods,
but his proof relies on a number of highly technical microstates estimates.
Our proof is considerably shorter than Jung's, avoiding much of the
difficulty of dealing with matricial microstates. Indeed, we instead
produce estimates for the non-microstates Fisher information and non-microstates
free entropy and then use results of \cite{guionnet-biane-capitaine:largedeviations}
to deduce the corresponding microstates free entropy inequality and
strong boundedness of free entropy dimension. Finally, we apply our
estimate to the case of sofic groups.

\subsection*{Acknowledgments.}

I would like to thank Institut Mittag-Leffler and the Hausdorff Institute for Mathematics, where parts of this work were completed, for their support and hospitality. I am
also grateful to Ian Charlesworth and Ben Hayes for a number of useful
discussions.

\section{Estimates on Fisher Information and Free Entropy.}

\subsection{Preliminaries and notation.}

Throughout this section we write $X=(X_{1},\dots,X_{n})$ for an $n$-tuple
of self-adjoint variables in a tracial von Neumann algebra $(M,\tau)$.
Thus we write $W^{*}(X)$ for $W^{*}(X_{1},\dots,X_{n})$, etc. If
$S=(S_{1},\dots,S_{n})$ is another $n$-tuple we write $X+\sqrt{\epsilon}S$
for the $n$-tuple $(X_{1}+\sqrt{\epsilon}S_{1},\dots,X_{n}+\sqrt{\epsilon}S_{n})$.
We also write $\Vert X\Vert_{2}=(\sum\Vert X_{j}\Vert_{2}^{2})^{1/2}$.

We write $F=(F_{1},\dots,F_{k})$ for a vector-valued non-commutative
polynomial function of $n$ variables. Here each $F_{j}\in\mathbb{C}[t_{1},\dots,t_{n}]$
is a non-commutative polynomial. We write $\partial F$ for the $k\times n$
matrix $(\partial_{j}F_{i})_{ij}\in M_{k\times n}(\mathbb{C}[t_{1},\dots,t_{n}]\otimes\mathbb{C}[t_{1},\dots,t_{n}])$,
where $\partial_{j}$ are the non-commutative difference quotient
derivations determined by the Leibniz rule and $\partial_{j}t_{i}=\delta_{i=j}1\otimes1$.
We write $F(X)$, $\partial F(X)$, etc. when we evaluate these functions
by substituting $X=(X_{1},\dots,X_{n})$ for $(t_{1},\dots,t_{n})$. 

We also view $\partial F(X)$ as a linear operator in $M_{k\times n}(W^{*}(X)\bar{\otimes}W^{*}(X)^{o})$.
It lies in the commutant of the von Neumann algebra $W^{*}(X)\bar{\otimes}W^{*}(X)^{o}$
acting on $L^{2}(W^{*}(X)\otimes W^{*}(X)^{o})$ by $(a\otimes b)\cdot(\xi\otimes\eta)=\xi a\otimes b\eta$.
The rank  of $\partial F(X)$ is the Murray-von Neumann
dimension (over the algebra $W^{*}(X)\bar{\otimes}W^{*}(X)^{o}$ of
the closure of the image of $\partial F(X)$. 

Finally, for a tensor $Q=a\otimes b$ and an element $x$ we write
$Q\#x$ for $axb$. We use the same notation for the multiplication
in the von Neumann algebra $W^{*}(X)\bar{\otimes}W^{*}(X)^{o}.$ We
also use the same notation for $n$-tuples and matrices: if $Q=(Q_{ij})$
is an $l\times k$ matrix and $Y=(Y_{1},\dots,Y_{k})$ is a $k$-tuple,
we write $Q\#Y$ for the $l$-tuple $(\sum_{j=1}^{k}Q_{ij}\#Y_{j})_{i=1}^{l}$.
With this notation we have the following perturbative formula: $F(X+\sqrt{\epsilon}S)=F(X)+\sqrt{\epsilon}\partial F\#S+O(\epsilon)$;
it is sufficient to check its validity for monomial $F$, which is
straightforward.

\subsection{The main estimate on Fisher information $\Phi^{*}$. }

In this paper we will denote by $\log_{+}(t)$ the function given
by $\log_{+}(t)=t$ for $t>0$ and $\log_{+}(0)=0$. 
\begin{lem}
\label{lem:FisherInequality}Suppose that $F=(F_{1},\dots,F_{k})\in\mathbb{C}[t_{1},\dots,t_{n}]^{k}$
is a vector-valued polynomial non-commutative function of $n$ variables,
and suppose that for some $X=(X_{1},\dots,X_{n})\in(M,\tau)$, $F(X)=0$.
Assume that $\log_{+}[\partial F(X)^{*}\partial F(X)]\in L^{1}(W^{*}(X)\otimes W^{*}(X)^{o})$.

Let $S=(S_{1},\dots,S_{n})$ be a free semicircular family, free from
$X$. Then there exist $\epsilon_{0}>0$ and $f\in L^{1}[0,\epsilon_{0}]$
so that the free Fisher information satisfies the inequality
\[
\Phi^{*}(X+\sqrt{\epsilon}S)\geq\frac{\operatorname{rank}\partial F(X)}{\epsilon}+f(\epsilon),\qquad0<\epsilon<\epsilon_{0}.
\]
\end{lem}
\begin{proof}
Denote by $E_{\epsilon}$ the conditional expectation from $W^{*}(X,S)$
to $W^{*}(X+\sqrt{\epsilon}S)$, and denote by the same letter the
extension of $E_{\epsilon}$ to $L^{2}$. Then (see \cite{dvv:entropy5})
\[
\Phi^{*}(X+\sqrt{\epsilon}S)=\Vert\xi_{\epsilon}\Vert_{2}^{2}
\]
where
\[
\xi_{\epsilon}=\epsilon^{-1/2}E_{\epsilon}(S).
\]

Since $F(X)=0$, using Taylor expansion we have that $F(X+\sqrt{\epsilon}S)=\sqrt{\epsilon}\partial F\#S+\zeta_{1}(\epsilon)$
with $\Vert\zeta_{1}(\epsilon)\Vert_{2}=O(\epsilon)$. Thus, using
contractivity of $E_{\epsilon}$ in $L^{2}$, we deduce: 
\begin{eqnarray}
E_{\epsilon}(\sqrt{\epsilon}\partial F(X)\#S) & = & E_{\epsilon}(F(X+\sqrt{\epsilon}S))+E_{\epsilon}(\zeta_{1}(\epsilon))\nonumber \\
 & = & F(X+\sqrt{\epsilon}S)+\zeta_{2}(\epsilon)\label{eq:eq1}\\
 & = & \sqrt{\epsilon}\partial F(X)\#S+\zeta_{3}(\epsilon)\nonumber 
\end{eqnarray}
with $\Vert\zeta_{j}(\epsilon)\Vert_{2}=O(\epsilon)$, $j=2,3$. Setting
$X_{\epsilon}=X+\sqrt{\epsilon}S$, we also have:
\begin{eqnarray*}
E_{\epsilon}(\sqrt{\epsilon}\partial F(X)\#S) & = & \sqrt{\epsilon}E_{\epsilon}(\partial F(X_{\epsilon})\#S)+\zeta_{4}(\epsilon)\\
 & = & \sqrt{\epsilon}\partial F(X_{\epsilon})\#E_{\epsilon}(S)+\zeta_{4}(\epsilon)\\
 & = & \sqrt{\epsilon}\partial F(X)\#E_{\epsilon}(S)+\zeta_{5}(\epsilon)
\end{eqnarray*}
where once again $\Vert\zeta_{j}(\epsilon)\Vert_{2}=O(\epsilon)$,
$j=4,5$. Combining this with (\ref{eq:eq1}) gives
\begin{eqnarray*}
\sqrt{\epsilon}\partial F(X)\#E_{\epsilon}(S) & = & \sqrt{\epsilon}\partial F(X)\#S+\zeta_{6}(\epsilon)
\end{eqnarray*}
where $\Vert\zeta_{6}(\epsilon)\Vert_{2}=O(\epsilon)$. Since $\xi_{\epsilon}=\epsilon^{-1/2}E_{\epsilon}(S)$,
it follows that
\[
\partial F(X)\#\xi_{\epsilon}=\epsilon^{-1/2}\partial F(X)\#S+\zeta_{\epsilon}',
\]
where $\Vert\zeta_{\epsilon}'\Vert_{2}=O(1)$. 

Let $E'$ be the orthogonal projection onto the $L^{2}$ closure of
$H=\operatorname{span}(W^{*}(X)SW^{*}(X))\subset L^{2}(W^{*}(X,S))$.
The map $Q\mapsto Q\#S$ is an isometric isomorphism of $L^{2}(W^{*}(X)\otimes W^{*}(X)^{o})$
with $H$ (as is easily verified by direct computation based on the
freeness condition). 

Since $H$ is invariant under left and right multiplication by variables from $X$,  $E'$ commutes with $\partial F(X)\#\cdot$ and thus
\[
\partial F(X)\#E'(\xi_{\epsilon})=\epsilon^{-1/2}\partial F(X)\#S+\zeta_{\epsilon}''
\]
with $\Vert\zeta''_{\epsilon}\Vert_{2}=O(1)$ and $\zeta_\epsilon''\in H$.
Applying $(\partial F)^{*}\#$ to both sides and denoting $(\partial F)^{*}\#(\partial F)$
by $Q$ finally gives us that there exists an $0<\epsilon_{0}<1$
and a constant $K$ for which
\[
Q\#E'(\xi_{\epsilon})=\epsilon^{-1/2}Q\#S+\zeta_{\epsilon},\qquad0<\epsilon<\epsilon_{0}
\]
with $\zeta_{\epsilon}\in H$, $\Vert\zeta_{\epsilon}\Vert_{2}\leq K/2n^{1/2}$. 

Let $P_{\lambda}=\chi_{[\lambda,+\infty)}(Q)$ and denote by $R_{\lambda}$
the operator $f_{\lambda}(Q)$ where $f(x)=x^{-1}$ for $x\geq\lambda$
and $f(x)=0$ for $0\leq x<\lambda$. Then $R_{\lambda}Q=P_{\lambda}$.

Since projections are contractive on $L^{2}$, it follows that for
any $\lambda>0$, $\Vert P_{\lambda}\#S\Vert_{2}\leq\Vert S\Vert_{2}=n^{1/2}$
and furthermore 
\begin{eqnarray*}
\Vert\xi_{\epsilon}\Vert_{2}^{2} & \geq & \Vert P_{\lambda}E'(\xi_{\epsilon})\Vert_{2}^{2}\\
 & = & \Vert R_{\lambda}\#Q\#E_{1}(\xi_{\epsilon})\Vert_{2}^{2}\\
 & = & \Vert\epsilon^{-1/2}R_{\lambda}\#Q\#S+R_{\lambda}\#\zeta_{\epsilon}\Vert_{2}^{2}\\
 & \geq & \epsilon^{-1}\tau(P_{\lambda})-2\epsilon^{-1/2}|\langle P_{\lambda}\#S,R_{\lambda}\#\zeta_{\epsilon}\rangle|\\
 & \geq & \epsilon^{-1}\tau(P_{\lambda})-\epsilon^{-1/2}\lambda^{-1}K.
\end{eqnarray*}
Let $r$ be the rank of $\partial F(X)$ (i.e., $r=\lim_{\lambda\to0}\tau(P_{\lambda})$,
where $\tau$ denotes the non-normalized trace on $M_{n\times n}(W^{*}(X)\bar{\otimes}W^{*}(X)^{o})$),
and set $\lambda=\epsilon^{1/4}$. Let $\phi(\lambda)=\tau(1-P_{\lambda})$.
Then
\begin{eqnarray*}
\Vert\xi_{\epsilon}\Vert_{2}^{2} & \geq & \epsilon^{-1}(r-\phi(\lambda))-\epsilon^{-1/2}\lambda^{-1}K\\
 & = & \frac{r}{\epsilon}-\epsilon^{-1}\phi(\epsilon^{1/4})-\epsilon^{-3/4}K\\
 & = & \frac{r}{\epsilon}-\epsilon^{-1}\phi(\epsilon^{1/4})+f_{1}(\epsilon)
\end{eqnarray*}
 for some $f_{1}\in L^{1}[0,\epsilon_{0}]$. Thus to conclude the
proof, it is enough to show that $\epsilon^{-1}\phi(\epsilon^{1/4})$
is integrable on $[0,\epsilon_{0}]$. 

Let $d\mu(t)$ be the composition of the trace $\tau$ with the spectral
measure of $Q$, so that $\phi(\lambda)=\int_{0^{+}}^{\lambda}d\mu(t)$.
Let us substitute $u=\epsilon^{1/4}$ (so $du=(1/4)\epsilon^{-3/4}d\epsilon$,
$u_{0}=\epsilon_{0}^{1/4})$ into the integral expression for the
$L^{1}$ norm of $\epsilon^{-1}\phi(\epsilon^{1/4})$:
\begin{eqnarray*}
\int_{0^{+}}^{\epsilon_{0}}\frac{1}{\epsilon}\phi(\epsilon^{1/4})d\epsilon & = & \frac{1}{4}\int_{0^{+}}^{u_{0}}\frac{1}{u}\phi(u)du\\
 & = & \frac{1}{4}\int_{0^{+}}^{u_{0}}(\partial_{u}\log u)\int_{0}^{u}d\mu(t)\ du-\left[\frac{1}{4}\phi(u)\log u\right]_{0^{+}}^{u_{0}}\\
 & \leq & \frac{1}{4}\int_{0}^{u_{0}}|\log_{+}t|d\mu(t)+\frac{1}{4}\phi(u_{0})|\log u_{0}|+\frac{1}{4}\lim_{\delta\to0}|\log\delta|\phi(\delta).
\end{eqnarray*}
Since by assumption $\Vert\log_{+}Q\Vert_{1}<\infty$, the intergral
$\int_{0}^{u_{0}}|\log_{+}t|d\mu(t)$ is finite. Since if $\delta<1$,
$|\log\delta|\leq|\log t|$ for $t\in(0,\delta]$, 
\begin{eqnarray*}
|\log\delta|\phi(\delta) & = & \int_{0^{+}}^{\delta}|\log\delta|d\mu(t)\leq\int_{0}^{\delta}|\log_{+}t|d\mu(t)
\end{eqnarray*}
we conclude also that $\lim_{\delta\to0}|\log\delta|\phi(\delta)$
is finite. Thus $\epsilon^{-1}\phi(\epsilon^{1/4})\in L^{1}[0,\epsilon_{0}]$. \end{proof}
\begin{rem}
If one drops the assumption that $\log_{+}[\partial F(X)^{*}\partial F(X)]\in L^{1}(W^{*}(X)\otimes W^{*}(X))$,
the proof of Lemma \ref{lem:FisherInequality} yields the estimate
\[
\Phi^{*}(X+\sqrt{\epsilon}S)\geq\epsilon^{-1}\tau(P_{\lambda})-\epsilon^{-1/2}\lambda^{-1}K,
\]
where $\lambda$ is arbitrary. This readily implies that $\delta^{\star}(X)\leq n-\tau(P_{\lambda})$
for all $\lambda$, which shows that $\delta^{*}(X)\leq n-\operatorname{rank}\partial F(X)$.
This estimate can be easily obtained from \cite{connes-shlyakht:l2betti}
by noting that for any finite-rank operator $T\in FR$, 
\[
0=[F_{j}(X),T]=\sum_{k}[\partial_{k}F_{j}(X)\#T,X_{k}]
\]
showing that the dimension of the $L^{2}$ closure of the space
\[
\{(T_{1},\dots,T_{n})\in FR^{n}:\sum_{j}[T_{j},X_{j}]=0\}
\]
is at least the rank of $\partial F$ (compare \cite{shlyakht-charlesworth:noAtoms}). 
\end{rem}

\begin{rem}
The proof of Lemma \ref{lem:FisherInequality} suggests the following
precise expression for the short-time asymptotics of the conjugate
variable $\xi_{\epsilon}$: 
\[
\xi_{\epsilon}=\epsilon^{-1/2}F\#S+\xi_{\epsilon}'+O(\epsilon^{1/2})
\]
where $F=(F_{ij})_{i,j=1}^{n}$ is the projection onto the $L^{2}$
closure of the space $\{(T_{1},\dots,T_{n})\in FR^{n}:\sum_{j}[T_{j},X_{j}]=0\}$
and $\xi_{\epsilon}'=\partial_{\epsilon}^{*}(\delta_{ij}1\otimes1-F)$.
Indeed, the intuition is that $F$ is the projection onto the space
of $L^{2}$ cycles, which is perpendicular to the space of $L^{2}$
derivations ($=\ker F$). The kernel of $F$ is a
kind of ``maximal domain of definition'' of $\partial_{\epsilon=0}^{*}$
; then $\xi'_{\epsilon}=E_{\epsilon}(\partial_{\epsilon=0}^{*}((\delta_{ij}1\otimes1)_{ij}-F))$
is the ``bounded part'' of $\xi_{\epsilon}$ while $\epsilon^{-1/2}F\#S$
is the ``unbounded part'' (compare \cite{shlyakht:qdim}). However,
we were unable to prove this exact formula. 
\end{rem}

\subsection{Estimates on free entropy.}
\begin{lem}
\label{lem:EntropyInequality}Under the hypothesis of Lemma \ref{lem:FisherInequality},
there exists $K<\infty$, $\epsilon_{0}>0$, so that for all $0<\epsilon<\epsilon_{0}$,
the non-microstates free entropy satisfies:
\[
\chi^{*}(X+\sqrt{\epsilon}S)\leq(\log\epsilon^{\frac{1}{2}})(\operatorname{rank}\partial F(X))+K
\]
In particular, the microstates entropy also satisfies
\[
\chi(X+\sqrt{\epsilon}S:S)\leq\chi(X+\sqrt{\epsilon}S)\leq(\log\epsilon^{\frac{1}{2}})(\operatorname{rank}\partial F(X))+K.
\]
\end{lem}
\begin{proof}
Let $\epsilon_{0}$, $f$ be as in Lemma \ref{lem:FisherInequality};
set $K'=\Vert f\Vert_{L^{1}[0,\epsilon_{0}]}$. By definition \cite{dvv:entropy5},
up to a universal constant depending only on $n$,

\[
\chi^{*}(X+\sqrt{\epsilon}S)=\frac{1}{2}\int_{0}^{\infty}-\Phi^{*}(X+\sqrt{\epsilon+t}S)+\frac{n}{1+t}dt
\]
Thus up to a finite universal constant,
\begin{eqnarray*}
\chi^{*}(X+\sqrt{\epsilon}S) & = & \frac{1}{2}\int_{\epsilon}^{\epsilon_{0}}-\Phi^{*}(X+\sqrt{t}S)dt\\
 & \leq & \frac{1}{2}\int_{\epsilon}^{\epsilon_{0}}-\frac{\operatorname{rank}(\partial F(X))}{t}+K'\\
 & = & (\log\epsilon^{\frac{1}{2}})(\operatorname{rank}\partial F(X))-\frac{1}{2}\log\epsilon_{0}+K',
\end{eqnarray*}
where we used the inequality of Lemma \ref{lem:FisherInequality}
to pass from the first to the second line.

The inequality for the microstates entropy is due to the general inequalities
$\chi(Z:Y)\leq\chi(Z)\leq\chi^{*}(Z)$; the first is trivial and the
second is a deep result of Biane, Capitaine and Guionnet \cite{guionnet-biane-capitaine:largedeviations}.
\end{proof}

\subsection{$r$-boundedness.}

The main consequence of our estimates is the following Theorem, originally
due to Jung \cite[Theorem 6.9]{jung:BettiBounded}. We give a short
alternative proof based on our estimates for non-microstates entropy.
\begin{thm}
\label{thm:StronglyBounded}Suppose that $F=(F_{1},\dots,F_{k})$
is a vector-valued polynomial non-commutative function of $n$ variables,
and suppose that for some $X=(X_{1},\dots,X_{n})\in(M,\tau)$, $F(X)=0$.
Assume that $\log_{+}[\partial F(X)^{*}\partial F(X)]\in L^{1}(W^{*}(X)\otimes W^{*}(X))$.
Then $X$ is $n-\operatorname{rank}(\partial F(X))$ bounded.\end{thm}
\begin{proof}
By Lemma \ref{lem:EntropyInequality}, 
\[
\limsup_{\epsilon\to0}\chi(X+\sqrt{\epsilon}S:S)+\operatorname{rank}(\partial F(X))|\log\epsilon^{\frac{1}{2}}|\leq K<\infty.
\]
Thus by the last bullet point in \cite[Corollary 1.4]{jung:onebounded},
$X$ is $n-\operatorname{rank}(\partial F(X))$-bounded. \end{proof}
\begin{rem}
Motivated by \cite[Corollary 1.4]{jung:onebounded} one can make the
following definition: $X$ is  $r$-bounded for $\delta^{*}$ (resp.
$\delta^{\star}$) if for some $K$ and all $0<\epsilon<\epsilon_{0}$,
\[
\chi^{*}(X+\sqrt{\epsilon}S)\leq(n-r)\log\epsilon^{\frac{1}{2}}+K
\]
(respectively, $\Phi^{*}(X+\sqrt{\epsilon}S)\geq\epsilon^{-1}(n-r)+\phi(\epsilon)$,
$0<\epsilon<\epsilon_{0}$ with $\phi\in L^{1}[0,\epsilon_{0}]$).
Then under the hypothesis of Theorem \ref{thm:StronglyBounded}, $X$
is  $n-\operatorname{rank}(\partial F(X))$-bounded for $\delta^{*}$
and $\delta^{\star}$. 
\end{rem}

\section{Applications.}
\begin{lem}
\label{lemma:Betti}Let $\Gamma$ be a finitely generated finitely
presented group. Then there exists an $n$-tuple $X=(X_{1},\dots,X_{n})$
of self-adjoint elements in $\mathbb{Q}\Gamma\subset L\Gamma$ and
a vector-valued polynomial function $F=(F_{1},\dots,F_{k})$ with
$F_{j}\in\mathbb{Q}[t_{1},\dots,t_{n}]$ so that $F(X)=0$ and moreover
\[
\operatorname{rank}\partial F(X)=n-(\beta_{1}^{(2)}(\Gamma)+\beta_{0}^{(2)}(\Gamma)-1),
\]
where $\beta_{j}^{(2)}(\Gamma)$ are the $L^{2}$-Betti numbers of
$\Gamma$. \end{lem}
\begin{proof}
Let $g_{1},\dots,g_{m}$ be generators of $\Gamma$ and let $n=2m$.
Consider the following generators $X_{1},\dots,X_{n}$ for the group
algebra $\mathbb{C}\Gamma$: $X_{j}=g_{j}+g_{j}^{-1}$, $X_{m+j}=i(g_{j}-g_{j}^{-1})$,
$j=1,\dots,m$. 

Denote by $h_{1},\dots,h_{m}$ the generators of the free group $\mathbb{F}_{m}$,
and let $Y_{j}=h_{j}+h_{j}^{-1}$, $Y_{n+j}=i(h_{j}-h_{j}^{-1})$,
$j=1,\dots,m$ be generators of $\mathbb{C}\mathbb{F}_{m}$. Let $R_{1},\dots,R_{p}\in\mathbb{F}_{m}$
be the relations satisfied by the generators $g_{1},\dots,g_{m}$
of $\Gamma$. 

Denote by $t_{1},\dots,t_{n}$ the generators of the algebra $\mathscr{A}=\mathbb{C}[t_{1},\dots,t_{2n}]$
of non-commutative polynomials in $2n$ variables. Then there is a
canonical map from $\mathscr{A}\to\mathbb{C}\mathbb{F}_{m}$ given
by $t_{j}\mapsto Y_{j}$. The algebra $\mathbb{C}\mathbb{F}_{m}$
is then the quotient of $\mathscr{A}$ by the ideal $J_{0}$ generated
by the relations corresponding to the relations $h_{j}^{-1}h_{j}=h_{j}h_{j}^{-1}=1$
that hold in $\mathbb{C}\mathbb{F}_{m}$; written in terms of the
generators $t_{j}$ these relations take the form $F'_{j}=(t_{j}+t_{m+j})(t_{j}-t_{m+j})+4i$,
$F''_{j}=(t_{j}-t_{m+j})(t_{j}+t_{m+j})+4i$, $j=1,\dots,m$. 

The relations $R_{j}$ can be interpreted as polynomials (with rational
coefficients) in $t_{1},\dots,t_{n}$ by substituting $\frac{1}{2}(t_{j}+t_{m+j})$
for $h_{j}$ and $\frac{1}{2i}(t_{j}-t_{m+j})$ for $h_{j}^{-1}$.
Let $F'''_{j}=R_{j}-1$, $j=1,\dots,p$. Let $J$ be the ideal in
$\mathscr{A}$ generated by $J_{0}$ and $F'''_{j}$, $j=1,\dots,p$.
Then $\mathbb{C}\Gamma$ is precisely the quotient of $\mathscr{A}$
by $J$, the quotient map sending $t_{j}$ to $X_{j}$. Let $F=(F_{j}''')_{j=1}^{p}\sqcup(F_{j}'')_{j=1}^{n}\sqcup(F_{j}')_{j=1}^{n}$,
where $\sqcup$ refers to union of ordered tuples. 

Let $\delta':\mathscr{A}\to L^{2}(\mathbb{C}\Gamma\otimes\mathbb{C}\Gamma^{o})$
be a derivation. If for all $j=1,\dots,p+2n$, $\delta'(F_{j})=0$
then for any $x=aF_{j}b\in J$, 
\[
\delta'(x)=\delta'(a)(F_{j})b+a(F_{j})\delta'(b)+a\delta'(F_{j})b=0
\]
since $F_{j}$ acts by zero on both the right and the left of $L^{2}(\mathbb{C}\Gamma\otimes\mathbb{C}\Gamma^{o})$
and $\delta'(F_{j})=0$ by assumption; thus $\delta'(J)=0$. Conversely,
if $\delta'(J)=0$ then $\delta(F_{j})=0$ for all $j$. It follows
that $\delta'$ descends to a derivation $\delta:\mathbb{C}\Gamma\to L^{2}(\mathbb{C}\Gamma\otimes\mathbb{C}\Gamma^{o})$
iff $\delta'(F_{j})=0$ for all $j=1,\dots,p+2n$. 

Let now $Q_{l}\in L^{2}(\mathbb{C}\Gamma\otimes\mathbb{C}\Gamma^{o})$,
$l=1,\dots,n$. Then there exists a unique derivation $\delta':\mathscr{A}\to L^{2}(\mathbb{C}\Gamma\otimes\mathbb{C}\Gamma^{o})$
so that $\delta'(t_{l})=Q_{l}$ for all $l$. This derivation descends
to a derivation $\delta$ on $\mathbb{C}\Gamma$ (satisfying $\delta(X_{l})=Q_{l}$)
iff for all $j$,
\[
0=\delta'(F_{j})=\partial F_{j}(X)\#Q,
\]
i.e., $Q\in\ker\partial F(X)$. It follows that the space of derivations
$Z^{1}(\mathbb{C}\Gamma,L^{2}(\mathbb{C}\Gamma\otimes\mathbb{C}\Gamma^{o}))$
is isomorphic to $\ker\partial F(X)\subset(L^{2}(\mathbb{C}\Gamma\otimes\mathbb{C}\Gamma^{o}))^{\oplus n}$. 

Let $\alpha$ denote the action of $L(\Gamma\times\Gamma^{o})$ on
$L^{2}(\mathbb{C}\Gamma\otimes\mathbb{C}\Gamma^{o})$ given by $(g\times h)\cdot(\xi\otimes\eta)=\rho(g)\xi\otimes\rho(h)\eta$,
where $\rho$ is the right regular representation. If $\delta:\mathbb{C}\Gamma\to L^{2}(\mathbb{C}\Gamma\otimes\mathbb{C}\Gamma^{o})$,
then $\alpha_{x}\circ\delta$ is again a derivation, for any $x\in L(\Gamma\times\Gamma^{o})$.
It follows that $L(\Gamma\times\Gamma^{o})$ acts on the space of
these $L^{2}$ derivations; moreover, the isomorphism $Z^{1}(\mathbb{C}\Gamma,L^{2}(\mathbb{C}\Gamma\otimes\mathbb{C}\Gamma^{o}))\cong\ker\partial F(X)$
is $\alpha$-equivariant. 

It is easily seen that
\[
\dim_{L(\Gamma\times\Gamma^{o})}Z^{1}(\mathbb{C}\Gamma,L^{2}(\mathbb{C}\Gamma\otimes\mathbb{C}\Gamma^{o}))=\beta_{1}^{(2)}(\Gamma)-\beta_{0}^{(2)}(\Gamma)+1
\]
(see e.g. \cite{connes-shlyakht:l2betti,shlyakht-mineyev:freedim};
note that $\beta_{1}^{(2)}(\Gamma)$ is the dimension of the space
of derivations which are not inner, while $1-\beta_{0}^{(2)}(\Gamma)$
is the dimension of the space of inner derivations). 

Putting things together, we obtain that
\[
\beta_{1}^{(2)}(\Gamma)-\beta_{0}^{(2)}(\Gamma)+1=\dim_{L(\Gamma\times\Gamma^{o})}\ker\partial F(X)=n-\operatorname{rank}\partial F(X)
\]
as claimed. Note that by construction $F\in\mathbb{Q}[t_{1},\dots,t_{n}]$. \end{proof}
\begin{thm}
\label{thm:BettiBounded}Let $\Gamma$ be a finitely generated, finitely
presented sofic group, and let $r=\beta_{1}^{(2)}(\Gamma)-\beta_{0}^{(2)}(\Gamma)+1$
. Then there exists a set of generators for $\mathbb{C}\Gamma$ which
is  $r$-bounded. In particular, if $|\Gamma|=\infty$ and if $\beta_{1}^{(2)}(\Gamma)=0$, then $L(\Gamma)$
is a strongly $1$-bounded von Neumann algebra.\end{thm}
\begin{proof}
By Lemma \ref{lemma:Betti} there exists an $n$-tuple $X=(X_{1},\dots,X_{n})$
of self-adjoint elements in $\mathbb{Q}\Gamma\subset L\Gamma$ and
a vector-valued polynomial function $F=(F_{1},\dots,F_{k})$ with
$F_{j}\in\mathbb{Q}[t_{1},\dots,t_{n}]$ so that $F(X)=0$ and moreover
$\operatorname{rank}\partial F(X)=n-(\beta_{1}^{(2)}(\Gamma)-\beta_{0}^{(2)}(\Gamma)+1)$. 

Furthermore, since $\partial F(X)\in\mathbb{Q}(\Gamma\times\Gamma^{o})$
and $\Gamma$ (thus also $\Gamma\times\Gamma^{o})$ is sofic, the
determinant conjecture holds \cite{Elek:determinantConj,skandalis:detConj,luck:book}.
Thus $\log_{+}[\partial F(X)^{*}\partial F(X)]\in L^{1}(W^{*}(X)\otimes W^{*}(X))$. 

We may thus apply Theorem \ref{thm:StronglyBounded} to conclude that
$(X_{1},\dotso,X_{n})$ is $n-\operatorname{rank}(\partial F(X))=\beta_{1}^{(2)}(\Gamma)-\beta_{0}^{(2)}(\Gamma+)1=r$
 bounded. 

Thus if $\beta_{1}^{(2)}(\Gamma)=0$, it follows that $r=1$.

If $\Gamma$ is an infinite group, $\beta_{0}^{(2)}(\Gamma)=0$.  If  $\Gamma$ is amenable, it follows \cite{connes:injective} that $L(\Gamma)$ is  hyperfinite and thus
from the results of \cite{jung:onebounded} we get that $L(\Gamma)$ is strongly $1$-bounded. 

Assume that $\Gamma$ is non-amenable; then any finite generating set of $\Gamma$ is a non-amenability set (cf. \cite{hayes:oneBounded}).
  Since $\Gamma$ is sofic, its microstates spaces are non-empty. It follows from \cite{hayes:oneBounded} 
that $L(\Gamma)$ is strongly $1$-bounded.  \end{proof}
\begin{rem}
The proof of Theorem \ref{thm:BettiBounded} still goes through if
we assume that $\Gamma$ satisfies the determinant conjecture (or,
more precisely, that $\log_{+}[\partial F(X)^{*}\partial F(X)]\in L^{1}(W^{*}(X)\otimes W^{*}(X))$
for the specific function $F$ we are dealing with) and that $L(\Gamma)$
satisfies the Connes embedding conjecture. 
\end{rem}
It is an open question whether infinite property $(T)$ von Neumann algebras
are always strongly $1$-bounded, even in the group case, although
it is known that every generating set must have free entropy dimension
at most $1$ \cite{hlyakht-jung:freeEntropyPropertyT}. However, we can settle the case of finitely presented sofic property $(T)$ groups:
\begin{cor}
Let $\Gamma$ be an infinite finitely presented sofic property $(T)$ group. Then $L(\Gamma)$ is strongly
$1$-bounded. \end{cor}
\begin{proof}
Indeed, property $(T)$ implies that $\beta_{1}^{(2)}(\Gamma)=0$
and that the group is finitely generated (see e.g. \cite{luck:book}).
\end{proof}
As noted in \cite[Corollary 4.6]{dvv:improvedrandom}, the free entropy
dimension $\delta_{0}$ is an invariant of the group algebra $\mathbb{C}\Gamma$
of a discrete group. 
\begin{cor}
\label{cor:Rigidity}Let $M$ be a finite von Neumann algebra that is not
strongly $1$-bounded.\footnote{For example, $M=W^{*}(X)$ with $\delta_{0}(X)>1$,
e.g. $M=L(\Gamma)$ with $\delta_{0}(\mathbb{C}\Gamma)>1$, e.g.,
$\Gamma=\mathbb{F}_{n}$, $n\geq2$.} If $M\cong L(\Lambda)$ with
$\Lambda$ finitely generated finitely presented sofic, then $\beta_{1}^{(2)}(\Lambda)\neq0$. \end{cor}
\begin{proof}
If $\beta_{1}^{(2)}(\Lambda)=0$, Theorem \ref{thm:BettiBounded}
would imply that $L(\Lambda)\cong M$ is strongly $1$-bounded. Contradiction. \end{proof}
\begin{conjecture}
\label{conj:FreeEntropyBetti}Let $\Gamma$ be a finitely presented
finitely generated sofic group with $\beta_{1}^{(2)}(\Gamma)>0$. Then\\
(a) $L(\Gamma)$ is not strongly one bounded;\\
(b) $\delta_0(\mathbb{C}\Gamma) > 1$;\\
(c) $\delta_{0}(\mathbb{C}\Gamma)=\beta_{1}^{(2)}(\Gamma)+1.$
\end{conjecture}
Note that (c)$\implies$(b)$\implies$(a) since if $L(\Gamma)$ were strongly $1$-bounded,
then $\delta_{0}(\mathbb{C}\Gamma)$ would have to be $1$. The inequality
$\delta_{0}(\mathbb{C}\Gamma)\leq\beta_{1}^{(2)}(\Gamma)+1$ is known
(see \cite{connes-shlyakht:l2betti}) and actually follows from the
results of the present paper as well. However, the reverse inequality occurring in statement (c)
is only known in a few special cases, e.g. $\Gamma=\mathbb{F}_{n}$, or free products of groups with vanishing first Betti number (possibly with amalgamation over finite or amenable groups)  \cite{brown-dykema-jung:amalgamated,shlyakht:lowerEstimates}. Nonetheless, it is possible that parts (a) or  (b) are more approachable than part (c). 
\begin{rem}
Together with Corollary \ref{cor:Rigidity}, Conjecture \ref{conj:FreeEntropyBetti}(a) 
would imply the following: if $\Gamma$ and $\Lambda$ are two finitely
generated finitely presented sofic groups and $L(\Gamma)\cong L(\Lambda)$ then $\beta_{1}^{(2)}(\Gamma)$,
$\beta_{1}^{(2)}(\Lambda)$ are either simultaneously zero or simultaneously
non-zero. In other words, the vanishing of $\beta_{1}^{(2)}$ is a
$W^{*}$-equivalence invariant for such groups.
\end{rem}
\bibliographystyle{plain}
%\bibliography{amsj,tex/quasifree}

\begin{thebibliography}{10}

\bibitem{skandalis:detConj}
G.~Balci and G.~Skandalis.
\newblock Traces on group {$C^*$}-algebras, sofic groups and {L\"uck}'s
  conjecture.
\newblock Preprint arXiv:1501.05753, 2015.

\bibitem{guionnet-biane-capitaine:largedeviations}
P.~Biane, M.~Capitaine, and A.~Guionnet.
\newblock Large deviation bounds for matrix {B}rownian motion.
\newblock {\em Invent. Math.}, 152(2):433--459, 2003.

\bibitem{shlyakht-charlesworth:noAtoms}
I.~Charlesworth and D.~Shlyakhtenko.
\newblock Regularity of polynomials in free variables.
\newblock Preprint arXiv:1408.0580v2, 2015.

\bibitem{connes:injective}
A.~Connes.
\newblock {Classification of injective factors. {C}ases {$II\sb{1},$}
                       {$II\sb{\infty },$} {$III\sb{\lambda },$} {$\lambda \not=1$}.}
\newblock {\em Ann. of Math.}, 104:73--115, 1976.

\bibitem{connes-shlyakht:l2betti}
A.~Connes and D.~Shlyakhtenko.
\newblock {$L^2$-homology for {von Neumann} algebras}.
\newblock {\em J. Reine Angew. Math.}, 586:125--168, 2005.

\bibitem{Elek:determinantConj}
G.~Elek and E.~Szab{\'o}.
\newblock Hyperlinearity, essentially free actions and l2-invariants. the sofic
  property.
\newblock {\em Mathematische Annalen}, 332(2):421--441, 2005.

\bibitem{hayes:oneBounded}
B.~Hayes.
\newblock $1$-bounded entropy and regularity problems in von {Neumann}
  algebras.
\newblock Preprint arXiv:1505.06682, 2015.

\bibitem{jung:dimHyperfinite}
K.~Jung.
\newblock A hyperfinite inequality for free entropy dimension.
\newblock {\em Proc. Amer. Math. Soc.}, 134(7):2099--2108 (electronic), 2006.

\bibitem{jung:onebounded}
K.~Jung.
\newblock Strongly 1-bounded von {N}eumann algebras.
\newblock {\em Geom. Funct. Anal.}, 17(4):1180--1200, 2007.

\bibitem{jung:BettiBounded}
K.~Jung.
\newblock The rank theorem and {$L^2$}-invariants in free entropy: global upper
  bounds.
\newblock Preprint arXiv:1602.04726, 2016.

\bibitem{brown-dykema-jung:amalgamated}
N.~Brown, K.~Dykema and K.~Jung.
\newblock Free entropy dimension in amalgamated free products.
\newblock {\em Proc. London Math. Soc.} 97(3): 339Ð367, 2008.

\bibitem{hlyakht-jung:freeEntropyPropertyT}
K.~Jung and D.~Shlyakhtenko.
\newblock Any generating set of an arbitrary property {$T$} von {N}eumann
  algebra has free entropy dimension {$\le1$}.
\newblock {\em J. Noncommut. Geom.}, 1(2):271--279, 2007.

\bibitem{luck:book}
W.~L{\"u}ck.
\newblock {\em {$L\sp 2$}-invariants: theory and applications to geometry and
  {$K$}-theory}, volume~44 of {\em Ergebnisse der Mathematik und ihrer
  Grenzgebiete. 3. Folge. A Series of Modern Surveys in Mathematics [Results in
  Mathematics and Related Areas. 3rd Series. A Series of Modern Surveys in
  Mathematics]}.
\newblock Springer-Verlag, Berlin, 2002.

\bibitem{shlyakht-mineyev:freedim}
I.~Mineyev and D.~Shlyakhtenko.
\newblock Non-microstates free entropy dimension for groups.
\newblock {\em Geom. Func. Anal.}, 15:476--490, 2005.

\bibitem{shlyakht:qdim}
D.~Shlyakhtenko.
\newblock Some estimates for non-microstates free entropy dimension, with
  applications to $q$-semicircular families.
\newblock {\em IMRN}, 51:2757--2772, 2004.

\bibitem{shlyakht:lowerEstimates}
D.~Shlyakhtenko.
\newblock Lower estimates on microstates free entropy dimension.
\newblock {\em Analysis and PDE}, 2:119--146, 2009.

\bibitem{dvv:entropy2}
D.-V. Voiculescu.
\newblock The analogues of entropy and of {Fisher's} information measure in
  free probability theory {II}.
\newblock {\em Invent. Math.}, 118:411--440, 1994.

\bibitem{dvv:entropy3}
D.-V. Voiculescu.
\newblock The analogues of entropy and of {Fisher}'s information measure in
  free probability theory, {III}.
\newblock {\em Geometric and Functional Analysis}, 6:172--199, 1996.

\bibitem{dvv:entropy5}
D.-V. Voiculescu.
\newblock The analogues of entropy and of {Fisher}'s information measure in
  free probability, {V}.
\newblock {\em Invent. Math.}, 132:189--227, 1998.

\bibitem{dvv:improvedrandom}
D.-V. Voiculescu.
\newblock A strengthened asymptotic freeness result for random matrices with
  applications to free entropy.
\newblock {\em IMRN}, 1:41 -- 64, 1998.

\bibitem{dvv:entropysurvey}
D.-V. Voiculescu.
\newblock Free entropy.
\newblock {\em Bull. London Math. Soc.}, 34(3):257--278, 2002.

\end{thebibliography}

\end{document}